\documentclass[11pt,reqno]{amsart}
\usepackage{amssymb,latexsym}
\usepackage[mathscr]{eucal}
 \usepackage{hyperref, flafter,epsf}
\usepackage{amsmath, amsthm, amsfonts, amscd, amssymb}
\usepackage{latexsym}
\usepackage[dvips]{graphics}
 \usepackage{fancyhdr}
 \usepackage{times}
 \usepackage[T1]{fontenc}
\usepackage{epsfig}
\usepackage{color}
 \usepackage{mathrsfs} 
\usepackage{verbatim}
\usepackage{epstopdf}
\input amssym.def
\input amssym.tex

\newtheorem{theorem}{Theorem}
\newtheorem{lemma}[theorem]{Lemma}
\newtheorem{corollary}[theorem]{Corollary}

\newcommand{\be}{\begin{equation}}
\newcommand{\ee}{\end{equation}}
\newcommand{\Lu}{\mathcal{L}(u,\chi_{D})}
\newcommand{\Lubar}{\mathcal{L}(\overline{u},\chi_{D})}
\newcommand{\Lubarzero}{\mathcal{L}(\overline{u}_0,\chi_{D})}
\newcommand{\Luzero}{\mathcal{L}(u_0,\chi_{D})}

\newcommand{\Fu}{\mathcal{F}(u,\chi_{D})}
\newcommand{\Fuzero}{\mathcal{F}(u_0,\chi_{D})}
\newcommand{\FKu}{\mathcal{F}_K(u,\chi_{D})}

\newcommand{\PKu}{ P_K(u,\chi_{D})}
\newcommand{\ZKu}{ Z_K(u,\chi_{D})}
\newcommand{\PKubar}{ P_K(\overline{u},\chi_{D})}

\renewcommand{\a}{\alpha}

\renewcommand{\th}{\theta} 
\newcommand{\s}{\sigma}

\newcommand{\e}{\epsilon}

 \usepackage{parskip}
 \usepackage{graphicx}
\usepackage{hyperref}

\renewcommand{\Re}{\operatorname{Re}}
\renewcommand{\Im}{\operatorname{Im}}
\usepackage[OT2,OT1]{fontenc}
\newcommand\cyr{%
\renewcommand\rmdefault{wncyr}%
\renewcommand\sfdefault{wncyss}%
\renewcommand\encodingdefault{OT2}%
\normalfont
\selectfont}
\DeclareTextFontCommand{\textcyr}{\cyr}

\title[Truncated Product Representations for $L$-Functions]
{Truncated Product Representations for $L$-Functions in the Hyperelliptic Ensemble}

 \author{J. C. Andrade, S. M. Gonek, and J. P. Keating}

\date{September 9, 2016}

\address{Department of Mathematics, University of Exeter, North Park Road, 
Exeter EX4 4QF, UK}
\email{j.c.andrade@exeter.ac.uk}

\address{Department of Mathematics, University of Rochester, Rochester, NY 14627, USA}
\email{gonek@math.rochester.edu}

\address{School of Mathematics, University of Bristol, Bristol BS8 1TW, UK}
\email{j.p.keating@bristol.ac.uk}

 \thanks{JCA would like to thank the Department of Mathematics at the University of Rochester for hospitality during a visit where this work started. SMG was supported in part by National Science Foundation Grant  DMS-1200582. JPK gratefully acknowledges support under EPSRC Programme Grant EP/K034383/1 (LMF: L-Functions and Modular Forms) and a Royal Society Wolfson Research 
Merit Award. The authors also would like to thank the American Institute of Mathematics (AIM) where this work was completed.}

 \subjclass[2010]{Primary 11G20, Secondary 14G10, 11M50}

 \keywords{hybrid formula, hyperelliptic curve, finite fields, function fields, random matrix theory, zeros of L-functions.}
\begin{document} 

 \maketitle

\begin{abstract} 
We investigate the approximation of quadratic Dirichlet $L$-functions  over function fields by 
truncations of their Euler products. We first establish   representations for such   $L$-functions   
as  products over prime polynomials  times   products over their zeros. This is the hybrid formula 
in function fields.  We then prove that partial Euler products are good approximations  of an 
$L$-function away from its zeros, and that,  when the length  of the product  tends  to infinity, 
we recover the original $L$-function.  We also obtain   explicit expressions for the arguments of   
quadratic Dirichlet  $L$-functions  over function fields and for the arguments of their partial Euler 
products. 
In the second part of the paper we construct, for each quadratic Dirichlet $L$-function over a function field, 
an auxiliary function  based on the approximate functional equation that equals the $L$-function on the 
critical line. We also construct a parametrized family of approximations of these auxiliary functions, prove  the Riemann 
hypothesis holds for them, and that their zeros are related to those of the associated $L$-function. Finally, 
we estimate the counting function for the zeros of this family of approximations,  show that these  zeros cluster 
near those of the associated $L$-function, and that, when the parameter is not too large,  almost all the zeros of the 
approximations are simple.
\end{abstract}



\section{Introduction}

Let $\mathbb{F}_{q}$ be a finite field   with $q$ elements, where $q$ is odd,  and let 
$\mathbb{F}_{q}[x]$ be the polynomial ring over $\mathbb{F}_{q}$ in the variable $x$. 
We denote by $\mathcal H_{2g+1, q}$ the set of monic, square-free polynomials 
$D\in \mathbb{F}_{q}[x]$ of degree $2g+1$. This is the hyperelliptic ensemble of the title. 
Associated with each $D$ is a nontrivial quadratic Dirichlet character $\chi_D$, and a  
quadratic Dirichlet $L$-function,  which  is the same as the Artin $L$-function corresponding to the character $\chi_D$ of $\mathbb{F}_{q}(x)\big(\sqrt{D(x)}\big)$, where $\mathbb{F}_{q}(x)$ is the rational function field over $\mathbb{F}_{q}$.
These  functions will be described more fully in the next section, but in order to explain the contents of this paper, we introduce some 
of the basic notation now.   Excellent general references are Rosen~\cite{Ro} and Thakur~\cite{Th}.  

If $f$ is a nonzero polynomial in $\mathbb{F}_{q}[x]$, we define the norm of  $f$  to be 
$|f| =q^{\mathrm{deg} f}$. If $f=0$, we set $|f|=0$. A monic irreducible polynomial $P$ 
is called a \emph{prime polynomial} or simply a \emph{prime}. 
The   $L$-function corresponding to the quadratic character $\chi_{D}$ is given by the Euler product
\begin{equation}\label{E Prod 1}
L(s,\chi_{D})=\prod_{P\,  \mathrm{prime}  }(1-\chi_{D}(P)|P|^{-s})^{-1} 
\qquad  \Re \, s>1,
\end{equation}
where $s$ is a complex variable.
Multiplying out, we obtain the Dirichlet series representation
\begin{equation}\label{D series 1} 
L(s,\chi_{D})=\sum_{ f\, \mathrm{monic} }  \frac{\chi_{D}(f) }{  |f|^{ s} } \qquad  \Re \, s>1.
\end{equation}
It is often convenient to work with the equivalent functions  written in terms of the variable $u=q^{-s}$, namely, 
\begin{equation}\label{E Prod 2}
\mathcal{L}(u,\chi_{D})=\prod_{ P\,  \mathrm{prime}   }
(1-\chi_{D}(P)u^{\mathrm{deg}P})^{-1} \qquad |u|<1/q  ,
\end{equation}
and
\begin{equation}\label{D Series 2} 
\mathcal{L}(u,\chi_{D})=\sum_{f\,\mathrm{monic} }\chi_{D}(f)u^{\mathrm{deg}f}
\qquad |u|<1/q.
\end{equation}
It turns out that $\Lu$ is actually a polynomial of degree $2g$ (see Rosen~\cite{Ro}, Proposition 4.3), and it satisfies a Riemann hypothesis  (see Weil~\cite{W}), namely,
all its zeros lie on the circle $|u| =q^{-\frac12}$.
 It follows that we may write
\be\label{L fnc roots def}
\Lu =\prod_{j=1}^{2g} (1-\a_j u),
\ee
where  the $\a_j = q^{\frac12}e(-\th_j), \,j=1,2\ldots,2g, $ are the reciprocals of the roots
$u_j=q^{-\frac12}e(\th_j)$ of $\Lu$. (Throughout we write  $e(x)$   to denote $e^{2\pi i x}$.)
In particular, the restriction $|u|<1/q$ in \eqref{D Series 2} (but not in 
\eqref{E Prod 2}) may be deleted.

Now  $\Lu$  satisfies the functional equation  
\be\label{Fnc Eqn} 
\Lu = (qu^2)^g \mathcal{L}(1/qu, \chi_D)
\ee 
and also possesses an ``approximate functional equation'' 
 \be\label{App Fnc Equ}
\mathcal{L}(u,\chi_{D})=
\sum_{\substack{f  \ \mathrm{monic}\\ \mathrm{deg}\, f\leq g}}
 \chi_{D}(f) u^{\deg f}+(qu^2)^{g}\sum_{\substack{f  \ \mathrm{monic}\\ \mathrm{deg}\, f\leq g-1}} \chi_{D}(f ) (qu)^{\deg f} ,
\ee
 which, of course, is exact. The name comes from the analogous formulas in the number field setting 
 which \emph{are}  approximations. For instance, for the Riemann zeta function, a symmetrized version of the formula is (see Titchmarsh~\cite{T})
 \be\label{AFE zeta} \notag
 \zeta(s) = \sum_{n\leq \sqrt{t/2\pi}} n^{-s} +\chi(s) \sum_{n\leq \sqrt{t/2\pi}} n^{s-1} +E(s),
 \ee
where $0<\Re\,s<1, \; t\geq1,$  and $E(s)$ is an error term. The importance of this   formula in applications is that it   consists of two Dirichlet polynomials  of length about 
$\sqrt t$, whereas a more direct  approximation (see Titchmarsh~\cite{T}) would require  a
Dirichlet polynomial  of length  $t$. The factor $\chi(s)$ is from the functional equation $\zeta(s)=\chi(s)\zeta(1-s)$ and is easy to calculate.
Similarly, \eqref{App Fnc Equ} consists of two pieces of length about $g$ as opposed to a   polynomial of length $2g$ (recall \eqref{L fnc roots def}). This is analogous because, in a sense, large $t$ 
in the number field case corresponds to $q^{2g}$.

In ~\cite{G1} and ~\cite{G2} another type of   approximation  of the Riemann zeta function and Dirichlet $L$-functions was constructed. 
It was based on the  approximate functional equation, but used truncations of the  $L$-function's Euler product  rather than its Dirichlet series. It was shown, for example,  that these approximations satisfy a Riemann hypothesis  and are very accurate if one stays away from the zeros of the $L$-function. Moreover, if the Riemann hypothesis holds for the $L$-function, the zeros of the approximations converge to those of the $L$-function as the length of the Euler product  tends to infinity.  This type of approximation has also been considered in the Physics literature; see, for example, ~\cite{BK1} 

Our goal in this paper is to carry out a similar construction and analysis in the case of 
quadratic $L$-functions for the hyperelliptic ensemble over finite function fields. 
An advantage we have in this setting is that the Riemann hypothesis is known to hold for such $L$-functions.
This means that all our results are unconditional.

The contents of the paper fall into two parts.
The first begins in Section 2  where we give some  background on quadratic characters and $L$-functions, and then prove a hybrid formula  for $\Lu$ (Theorem~\ref{L: hybrid form}). By this we mean  a representation of $\Lu$ as a product over prime polynomials times a product over its zeros. 
In Section 3 we prove that partial Euler products $\PKu$ approximate $\Lu$ well inside the disk
$|u| \leq q^{-\frac12}$ when $u$ is not close to any zero $u_j, \ j=1, 2, \ldots, 2g,$ of $\Lu$, and that 
at every point in the disc except the $u_j$'s, $\lim_{K\to\infty}\PKu=\Lu$. In Section 4 we obtain explicit expressions for
$\arg\Lu$ and $\arg\PKu$ and bound their difference. In the following section we reprove a recent estimate for $\arg\Lu$ of Faifman and Rudnick~\cite{FR}, and show that if $K$ is sufficiently large, this bound holds for $\arg\PKu$ as well. We also reprove, in a slightly different way, another result from~\cite{FR},   an  estimate for the counting function $N(\th, \chi_D)$ of the zeros of $\Lu$ on the arc $q^{-\frac12} e(\phi), \, 0\leq \phi\leq \th \leq 1$.

The second part of the paper begins with Section 6. 
We introduce an auxiliary function $\Fu$ modeled on the approximate functional equation \eqref{App Fnc Equ}
which equals $\Lu$ on the all important circle $|u|=q^{-\frac12}$ and has the same zeros as $\Lu$ in the complex plane. In Section 7 we construct a model   $\FKu$ of $\Fu$ built from the truncated Euler products $\PKu$.
We then show that the Riemann hypothesis holds for $\FKu$, that inside 
$|u|\leq q^{-\frac12}$, $\FKu$ approximates $\Fu$  well when $u$ is away from zeros $u_j$ of $\Fu$ and $K$ is large enough, and that in this disk $\lim_{K\to\infty}\FKu=\Fu$ if $u$ is not a $u_j$. Finally, in the eighth section we 
estimate the counting function $N_K(\th, \chi_D)$ of the zeros of $\FKu$, show that the zeros of $\FKu$ cluster around the zeros of $\Fu$ as $K\to\infty$, and show that when $K$ is not too large, almost all the zeros of $\FKu$ are simple.

In this paper, our main interest is when the cardinality $q$ of the ground field $\mathbb{F}_{q}$ is fixed and the genus $g$ gets large, i.e., $\text{deg}(D)\rightarrow\infty$. An interesting question is if the same analysis of this paper can be carried out with $g$ fixed and $q\rightarrow\infty$ and this is material for a future work.    

\section{Background on $L$-functions and a hybrid formula for $\mathcal L(u, \chi_D)$} 
 
For a prime polynomial $P$ and any $f\in \mathbb{F}_{q}[x]$, the quadratic residue symbol $\displaystyle \Big(\frac f P\Big)$ is defined by 
$$
 \Big(\frac{f}{P}\Big)=\left\{
\begin{array}{cl}
0, & \mathrm{if}\ P\mid f,\\
1, & \mathrm{if}\ P\not{|} f \ \mathrm{and} \ f \ \mathrm{is \ a \ square \ modulo} \ P,\\
-1, & \mathrm{if}\ P\not{|} f \ \mathrm{and} \ f \ \mathrm{is \ a \ non\ square \ modulo} \ P.\\
\end{array}
\right.
$$
If $Q=  P_{1}^{e_{1}}P_{2}^{e_{2}}\ldots P_{k}^{e_{k}}$ is the prime factorization of a  monic polynomial $Q \in \mathbb F_q[x] $, the Jacobi symbol is defined as
$$\Big(\frac{f}{Q}\Big)=\prod_{i=1}^{k}\Big(\frac{f}{P_{i}}\Big)^{e_{i}}.$$
If $c\in \mathbb{F}_{q}^{*}$, then
$$
\Big(\frac{c}{Q}\Big)= c^{((q-1)/2) \, {\mathrm deg}\,  Q}.
$$
If $A$ and $B$ in $\in\mathbb{F}_{q}[x]$ are monic  coprime polynomials, the quadratic reciprocity law, proved by E. Artin, says that
$$
\Big(\frac{A}{B}\Big) = \Big(\frac{B}{A}\Big) (-1)^{((q-1)/2){\mathrm{deg}} A\,  {\mathrm{deg}} B}.
$$
This also holds for $A, B$ not coprime as then both sides equal zero.
 
For  $D\in\mathbb{F}_{q}[x]$   monic and  square-free, we define the \textit{quadratic character} $\chi_{D}$ by
\begin{equation}\notag
\chi_{D}(f)=\Big(\frac{D}{f}\Big).
\end{equation}
 
For each such character there corresponds an   $L$-function  
(see \eqref{E Prod 1}-\eqref{L fnc roots def} above)
\begin{equation}\notag
L(s,\chi_{D})=\sum_{f\,  \mathrm{monic} }  \frac{\chi_{D}(f) }{  |f|^{ s} }
=\prod_{P\,  \mathrm{prime} }(1-\chi_{D}(P)|P|^{-s})^{-1} 
\qquad  \Re \, s>1. 
\end{equation}

For each $D$ in  the hyperelliptic ensemble 
$$
\mathcal H_{2g+1, q}= \{\, D\in\mathbb{F}_{q}[x]\,:\, D\,  \hbox {monic and  square-free}, 
\deg\,D=2g+1 \, \},
$$ 
there is an  associated hyperelliptic curve   given in affine form by
$$
C_{D}:  y^{2}=D(x) .
$$
These curves are nonsingular and of genus $g$, and the $L$-function defined above is related to the zeta function of the curve $C_D$ as follows.    
Recall that if $C$ is a smooth, projective, connected curve of genus $g$ over $\mathbb F_q$, its zeta function is defined as
$$
Z_{C}(u) =\exp\bigg( \sum_{r=1}^\infty N_r(C ) \frac{u^r}{r}    \bigg),
$$
where $N_r(C )$ is the number of points on $C$ with coordinates in 
$\mathbb F_{q^r}$ (including the point at infinity). Weil~\cite{W} proved that 
$$
Z_{C }(u) =\frac{P_{C}(u)}{(1-u)(1-qu)},
$$
where $P_{C }(u)$ is a polynomial of degree $2g$, and he proved the Riemann hypothesis for $Z_C(u)$, which states that all the zeros of $P_C(u)$ lie on the 
circle $|u| =q^{-\frac12}$.
In the case of our hyperelliptic curves $C_D$ of odd degree, it turns out that the polynomial 
$P_{C_D}(u)$ is exactly  $\Lu$ (this was first shown by Artin~\cite{A}).
As was mentioned above,  we may therefore write
\be\label{L fnc roots def 2}\notag
\Lu =\prod_{j=1}^{2g} (1-\a_j u) \qquad\quad u\in \mathbb C,
\ee
where  the $\a_j = q^{\frac12}e(-\th_j), \,j=1,2\ldots,2g, $ are the reciprocals of the roots
$u_j=q^{-\frac12}e(\th_j)$ of $\Lu$.

For a monic polynomial $f$ we write $\Lambda(f)= \deg  P$ if $f=P^k$ for some prime   $P$ and positive integer $k$, and $\Lambda(f)= 0$ otherwise. The logarithmic derivative of \eqref{E Prod 2} may then be written
\be\notag
\begin{split}
\frac{\mathcal L^{'}}{\mathcal L}(u, \chi_D) 
=&\sum_{P\, \mathrm{prime}   } 
\frac{(\mathrm{deg}\,P)  \chi_{D}(P)u^{\mathrm{deg}P -1}  }{1-\chi_{D}(P)
u^{\mathrm{deg}P}}\\
  =& \sum_{n=1}^\infty \Bigg(\sum_{\substack {f \,\mathrm{monic} \\ \deg\,f=n} } 
  \Lambda(f)\chi_D(f) \Bigg) u^{n-1}.
  \end{split}
\ee
On the other hand, the logarithmic derivative of \eqref{L fnc roots def} is
\be\notag
\frac{\mathcal L^{'}}{\mathcal L}(u, \chi_D )
= -\sum_{n=1}^\infty \Bigg(\sum_{j=1}^{2g}\a_j^{n}\Bigg)u^{n-1}.
\ee
Equating these two expressions,   
we find that  
\be\label{trace formula}
-\sum_{j=1}^{2g}e(-n\th_j  )= \frac{1}{q^{n/2}} 
\sum_{\substack{f\,\mathrm{monic}\\ \deg f=n} }\chi_D(f)\Lambda(f).
\ee
 

Using this fundamental formula, we prove a version of the hybrid formula of Gonek, Hughes, and Keating~\cite{GHK} (see also ~\cite{BK2} and ~\cite{BK3}) for $\Lu$.

\begin{theorem}[Hybrid formula for $\mathcal L(u, \chi_D)$]\label{L: hybrid form}
Let $K\geq 0$ be an integer and let
\be\label{P def}
P_K(u, \chi_D)=\exp\Bigg( \sum_{k=1}^{ K}\sum_{\substack{f\, {\rm{monic}}  \\ \deg f=k}} 
\frac{\Lambda(f) \chi_D(f) u^{k}}{k}\Bigg),
\ee
where $\Lambda(f)= \deg  P$ if $f=P^n$ for some prime polynomial $P$, and $\Lambda(f)= 0$ otherwise.  
Also set
\be\label{Z def}
Z_K(u, \chi_D)=\exp\bigg(- \sum_{j=1}^{2g} \bigg(\sum_{k>K} \frac{(\a_j u)^{k}}{k} \bigg)\bigg).
\ee
Then for $|u|\leq q^{-1/2}$, 
\be\label{E: hybrid}
\begin{split}
\mathcal L(u, \chi_D)= P_K(u, \chi_D)  \, Z_K(u, \chi_D).
\end{split}
\ee
\end{theorem}

\noindent{\bf Remark 1.} H. Bui and A. Florea~\cite{BF} have, independently and at the same time as the current authors, proved a slightly different (weighted) version of the hybrid formula. They use it to calculate low moments of the $L$-function along the lines of Gonek, Hughes and Keating~\cite{GHK}.

\noindent{\bf Remark 2.} One can prove  similar formulas for other $L$-functions 
defined over finite fields such as for all Dirichlet $L$-functions $L(s,\chi)$ with
$\chi$ a Dirichlet character modulo $Q\in\mathbb{F}_{q}[x]$.

\noindent{\bf Remark 3.}
We see that $\lim_{u\to u_j}Z_K(u, \chi_D)=0$  
within any sector $|\arg (u-u_j)|<\pi/2-\delta$, where
$0<\delta<\pi/2$ is fixed. Thus, we may interpret $Z_K(u_j, \chi_D), \, j=1, 2, \ldots, 2g,$ as zero, even though the infinite series defining $Z_K(u, \chi_D)$ does not converge at $u_j$. The reader should keep this convention in mind throughout. 

\noindent{\bf Remark 4.} If $K=0$, the sum defining $\PKu$ is empty, 
so we  interpret $P_0(u, \chi_D) $ as being identically $1$. 
Thus, $\Lu=Z_0(u,\chi_D)$ for $|u|\leq q^{-\frac12}$. 
Indeed, for such $u$
we see that
\be\label{L=Z_0}
\begin{split}
Z_0(u,\chi_D)
=&\exp\bigg(- \sum_{j=1}^{2g} \bigg(\sum_{k=1}^{\infty} 
\frac{( \a_j u)^{k}}{k} \bigg)  \bigg)\\
=&\exp\bigg( \sum_{j=1}^{2g} \log (1-  \a_j u) \bigg)
= \prod_{j=1}^{2g} (1- \a_j  u).
\end{split}
\ee
In the other   direction, for $|u|\leq q^{-\frac12}$ with small arcs around the zeros $u_j$ removed,  we see from \eqref{Z def} that 
$\lim_{K\to\infty} Z_K(u, \chi_D)=1$ uniformly. Hence, on this set
\be\notag
\begin{split}
\lim_{K\to\infty} P_K(u, \chi_D)
=&\exp\bigg( \sum_{k=1}^{\infty}\sum_{\substack{f\,{\rm{monic}} \\ \deg f=k}} 
\frac{\Lambda(f) \chi_D(f) u^{k}}{k}\bigg)\\
=&\exp\bigg( \sum_{P} \sum_{l=1}^\infty \frac{\chi_D(P^l) u^{l\deg P} }{l}\bigg)\\
=&\prod_{\substack{P \, \mathrm{prime} }}(1-\chi_{D}(P)u^{\mathrm{deg}P})^{-1}.
\end{split}
\ee
In other words, at the extremes, $K=0$ and $K=\infty$, we (essentially) recover  expressions for $\Lu$
as a product over zeros and a product over primes, respectively,  from the hybrid formula.

\begin{proof}
Assume first that
$|u| <q^{-\frac12}$.
Taking logarithms of both sides of  \eqref{L fnc roots def} and using the Taylor series for 
$-\log(1-z),$ \  $ |z|<1$, we see that 
\be\notag
\begin{split}
\log \mathcal L(u, \chi_D) 
=&-\sum_{j=1}^{2g} \Bigg( \sum_{k=1}^\infty \frac{(\a_j u)^{k} }{k} \Bigg)
=-\sum_{j=1}^{2g} \Bigg( \sum_{k=1}^K+\sum_{k=K+1}^\infty\Bigg) \frac{(\a_j u)^{k} }{k}  \\
=&-\sum_{k=1}^K   \frac{u^{k} }{k}  \bigg(\sum_{j=1}^{2g} \a_j^k \bigg)
- \sum_{j=1}^{2g}\Bigg(\sum_{k=K+1}^\infty  \frac{(\a_j u)^{k} }{k}\Bigg)\\
\end{split}
\ee
By  \eqref{trace formula}, the first double sum  equals
\be\notag
\begin{split}
\sum_{k \leq K} \frac{u^{k}}{k}
\Bigg(\sum_{\substack{f\, {\rm{monic}}  \\ \deg f=k}} \Lambda(f) \chi_D(f)  \Bigg).
\end{split}
\ee
The second is simply 
$$
 - \sum_{j=1}^{2g} \Bigg(\sum_{k>K} \frac{( q^{\frac12} e(-\theta_j ) u)^{k}}{k} \Bigg).
$$
Thus,  
\be\label{E: hybrid 2}\notag
\begin{split}
\log \mathcal L(u, \chi_D) 
=&\; \sum_{k=1}^{ K}\Bigg(\sum_{\substack{f\,{\rm{monic}}  \\ \deg f=k}} 
\frac{\Lambda(f) \chi_D(f) u^{k}}{k}\Bigg)
 - \sum_{j=1}^{2g} \Bigg(\sum_{k>K} \frac{( q^{\frac12} e(-\theta_j )u)^{k}}{k} \Bigg).
\end{split}
\ee
Exponentiating this, we obtain \eqref{E: hybrid} for $|u|<q^{-\frac12}$.   
It only remains   to treat the case $|u|=q^{-\frac12}$. On this circle the first term in \eqref{E: hybrid 2} 
is   a polynomial in $u$, so is continuous. The second term, with  $u=q^{-\frac12} e(\th)$ and  $\theta\neq \theta_j$, $j=1,2, \ldots , 2g$, equals
\be\label{Z exp}
\begin{split} 
   - \sum_{j=1}^{2g}\Bigg(\sum_{k>K} \frac{e(k(\th-\th_j))}{k} \Bigg)  .
\end{split}
\ee
By partial summation, the series
 $ 
\sum_{k>K}   e(k\phi)/k 
 $
converges uniformly for  $\delta\leq \phi \leq 1-\delta$, where $0<\delta<\frac12$ is fixed. It follows that \eqref{Z exp}   is continuous on the circle $|u|=q^{-\frac12}$ with the points $u_j$ deleted. Thus $\Lu$ and the function $P_K(u, \chi_D) \, Z_K(u, \chi_D)$ agree and are analytic  in $|u|<q^{-\frac12}$, and are continuous on the circle 
$|u|=q^{-\frac12}$ minus the points $u_1, \ldots, u_{2g}$. They therefore agree on this set, and by our interpretation of $Z_K$ as zero in the   limit as $u\to u_j$, they agree at these points as well. This completes the proof  of the theorem.
 \end{proof}



\section{Approximation of $\Lu$ by $\PKu$}
In this section it simplifies some expressions slightly if we use both the notations 
$\log g$ and  $\log_q g$.

Let $u=q^{-\s-it}$ with $\s>\frac12$ and assume that $K\geq 1$. Then  
\be\notag
\begin{split}
|\ZKu| = \bigg| \sum_{j=1}^{2g} \sum_{k>K} \frac{( \a_j u)^{k}}{k} \bigg| 
 \leq&  2g  \sum_{k>K} \frac{q^{  (\frac12-\s)k} }{k}  
 \leq   \frac{2g \  q^{  (\frac12-\s)K} }{K (1-q^{\frac12-\s}) }.
\end{split}
\ee
If we write 
 $\displaystyle
\s=\frac12+\frac{C\log_q g}{K}$,
with $C >0$ and   possibly depending on $g$, then this equals
$$
\frac{2g^{1-C} }{K   (1-g^{-C/K} ) }.
$$
Since $1-e^{-x}  \geq  {x}/{2}$ for $0\leq x\leq \frac12$,   if we  assume that 
$K\geq 2C\log g$  we have
$$
1-g^{-C/K} \geq \frac{C\log g}{2K}.
$$
Hence the above is
$$
\leq  \frac{4g^{1-C} }{ C\log g}.
$$
Using this with Theorem~\ref{L: hybrid form}, we see that if $C\geq 1$,  then
\be\notag
\Lu=\PKu \Big(1+O\Big(\frac{ 1 }{Cg^{C-1}\log g}\Big)\Big).
\ee

 \begin{theorem} \label{L approx P 1}
Let  $\displaystyle \s=\frac12+\frac{C\log_q g}{K}$ with $C\geq 1$ and
  $K\geq 2C\log g$.  Then for  $|u|\leq  q^{-\s}$ we have
\be\notag
\Lu=\PKu  \Big(1+O\Big(\frac{ 1 }{Cg^{C-1}\log g}\Big)\Big).
\ee
 \end{theorem}
 
We can prove a similar approximation  on the circle $|u|=q^{-1/2}$.
 Write $u=q^{-\frac12}e(\th)$. Then by \eqref{Z def}  
 \be\notag 
Z_K(e(\th)q^{-\frac12}, \chi_D)=\exp\bigg(- \sum_{j=1}^{2g} \bigg(\sum_{k>K} 
\frac{e(k(\th-\th_j ) }{k} \bigg)\bigg).
\ee
By partial summation
\be\notag
\sum_{j=1}^{2g} \bigg(\sum_{k>K} \frac{e(k(\th-\th_j ) }{k} \bigg)
\ll \frac{1}{K+1}  \sum_{j=1}^{2g}\frac{1}{|\sin(\pi (\th-\th_j))|}.
\ee
Let $\|x\| =\min_{n\in \mathbf Z}|x-n|$. Then if we assume that
$$
\min_{1\leq j\leq 2g}\|\th-\th_j\| \geq\frac{c}{2g }
$$
and use the estimate $|\sin \pi x|\geq 2\|x\|$, we find that the above is 
$$
\ll \frac{g^2}{cK} \ .
$$
Thus we have
 \begin{theorem}\label{L approx P 2}
Let  $u=q^{-\frac12}e(\th)$.  Suppose that $c>0$,
$$
\min_{1\leq j\leq 2g}\|\th-\th_j\| \geq \frac{c}{2g },
$$
 and $K  \geq g^2/c.$
 Then 
\be\notag
\Lu=\PKu \Big(1+O\Big(\frac{ g^2}{cK}\Big)\Big).
\ee
 \end{theorem}


Observe that as a consequence of Theorems~\ref{L approx P 1} and~\ref{L approx P 2}    
we have for  $|u|\leq q^{-\frac12}$, $u\neq u_j, \, j=1,2, \ldots, 2g$, that
$$
 \lim_{K\to\infty}\PKu\to \Lu.
$$
This was pointed out in {Remark 3} above,  but Theorems~\ref{L approx P 1} 
and ~\ref{L approx P 2} also supply rates of convergence.

\section{Explicit expressions for  $\arg \Lu$ and $\arg \PKu$}

A standard way to define the argument of   a quadratic Dirichlet $L$-function 
$L(\frac12+it, \chi_D)$
when  $t$ is not the ordinate of a  zero, is  by continuous variation starting with the value zero at  $s=2$, moving up the line $\sigma=2 $ as far as $2+it,$ and then horizontally over to $s=\frac12+it$. This makes sense for our  function field  $L$-functions defined in \eqref{E Prod 1} and \eqref{D series 1} as well.    For the alternate form $\Lu$ of $L(s, \chi_D)$, however,
this corresponds to continuous variation in the negative sense along the circular arc from $q^{-2}$ to   a point $q^{-2}e(-\th)$, 
and then along the radius  $ re(-\th)$ from $r=q^{-2}$ to $r=q^{-\frac12}e(-\th).$
This would be $\arg \mathcal L(q^{-\frac12}e(-\th),\chi_D)$, which is the same as\ $-\arg \mathcal L(q^{-\frac12}e(\th),\chi_D)$, since $\Lu$ is real for real values of $u$. 
Thus, denoting the path consisting of the positively oriented  circular arc from $q^{-2}$ to  $q^{-2}e(\th)$ 
followed by the radial segment  from  $q^{-2}e(\th)$ to $ q^{-\frac12}e(\th)$ 
 by  $\Gamma(\th)$, we define
\be\label{arg L 1}
\arg \mathcal L(q^{-\frac12}e(\th),\chi_D)= -\triangle_{\Gamma(\th)}\; \Lu.
\ee
If $q^{-\frac12}e(\th)$ happens to be a zero of $\Lu$, we use the convention that
\be\label{arg L at zeros}
\begin{split}
\arg \mathcal L( &q^{-\frac12}e(\th),\chi_D) 
= \lim_{\e\to 0^{+}}  \arg \mathcal L(q^{-\frac12}e(\th+\e),\chi_D).
\end{split}
\ee
We also define
\be\label{S def 1} 
S(\th,\chi_D) =\frac1\pi \arg\, \mathcal L(q^{-1/2}e(\th),\chi_D).
\ee  
Similarly we let
\be\label{arg PK}
\arg P_K(q^{-\frac12}e(\th),\chi_D)= -\triangle_{\Gamma(\th)}\; \PKu 
\ee
and 
\be\label{SK}\notag
S_K(\th,\chi_D) =\frac1\pi \arg\,   P_K(q^{-1/2}e(\th),\chi_D).
 \ee

Our next goal   is to obtain alternative expressions for   these arguments.
From \eqref{L fnc roots def} and \eqref{arg L 1} we find that if $\th\neq \th_j$ for any $j=1, 2, \ldots,2g$, then
\be\notag 
\begin{split}
 \arg \mathcal L(q^{-1/2}e(\th),\chi_D) 
= &-\triangle_{\Gamma(\th)} \arg \prod_{j=1}^{2g}(1-\a_j u)   \\
=&-\sum_{j=1}^{2g}  \triangle_{\Gamma(\th)}\arg(1-e(\th-\th_j)),
\end{split}
\ee
where, on the last line, we use the value of the argument in $(-\pi/2,\pi/2)$.
Elementary geometric reasoning shows that if $ \phi\notin \mathbb Z$,  then
$\arg(1-e(\phi))=\pi(\{\phi\}-\frac12)$, where $\{x\}$ denotes the fractional part of the real number $x$.
Thus
\be\notag
\begin{split}
 \arg \mathcal L(q^{-1/2}e(\th),\chi_D)
=&\pi \sum_{j=1}^{2g} \big((\{ -\th_j \}-\tfrac12 )-(\{\th-\th_j \}-\tfrac12 )\big)\\
=&\pi \sum_{j=1}^{2g} \big( \{ -\th_j \}-\{\th-\th_j \}\big).
\end{split}
\ee
It follows that for $\th\neq \th_j$,
\be\label{S formula 1}
S(\th,\chi_D) =\sum_{j=1}^{2g} \big(\{ -\th_j \} -\{\th-\th_j \} \big).
\ee
If $\th$ does equal $\th_i$ for some $i$, then by \eqref{arg L at zeros} and \eqref{S def 1},
\be \notag
\begin{split}
S(\th_i,\chi_D)=& \lim_{\e\to0^{+}}
\bigg(  \sum_{j=1}^{2g} (\{-\th_j\}-\{\th_i+\e-\th_j \}) \\
 =&  \sum_{\substack{j=1 \\  j \neq i}}^{2g}  \big(\{-\th_j\}-\{\th_i-\th_j \}\big)
 +  \lim_{\e\to0^{+}}    \big(\{-\th_i\}-\{\th_i+\e-\th_i \}\big)\\
=& \sum_{\substack{j=1}}^{2g}  \big(\{-\th_j\}-\{\th_i-\th_j \}\big).
 \end{split}
\ee
Thus \eqref{S formula 1} holds whether or not $\th$ is  a $\th_i$.

We can express the formula in both cases in a unified way by using the function
\be\notag
s(x)=
\begin{cases}
 \{x\}-\frac12 &\;\hbox{if}\; x\in \mathbb R\setminus \mathbb Z,\\
0 &\;\hbox{if}\; x\in \mathbb Z.
\end{cases}
\ee
We then clearly   have for all $\th$  that
\be\label{S formula 2}
S(\th,\chi_D)=
 \sum_{\substack{j=1  }}^{2g} \big(s(-\th_j)- s(\th -\th_j)\big).
\ee  
Notice that   since $0<\th_j<1$ for all $j=1,2,\ldots,2g$ and since $\th_{g+j}=1-\th_j$ 
for $j=1,2,\ldots,g$,  we have
\be \notag
\begin{split}
 \sum_{\substack{j=1  }}^{2g} s(-\th_j)
 = &\sum_{\substack{j=1  }}^{2g} \big(\{-\th_j\}-\tfrac12\big)
 =\sum_{\substack{j=1  }}^{2g} \big((1-\th_j)-\tfrac12\big)\\
 =&g-\sum_{\substack{j=1  }}^{2g}\th_j 
  = g-\sum_{\substack{j=1  }}^{g}\big(\th_j +(1-\th_j)\big) =0.
\end{split}
\ee
Thus, \eqref{S formula 2} is equivalent to
\be\notag
S(\th,\chi_D)=
- \sum_{\substack{j=1  }}^{2g}  s(\th -\th_j).
\ee
Now it is well  known  (for example, see Montgomery and Vaughan~\cite{MV}, p. 536) that
\be\label{sawtooth approx}
s(x) = -\sum_{k=1}^{K} \frac{\sin (2\pi x k)}{\pi k} +E_K(x),
\ee
where for $K\geq 1$,
\be\label{E_K}\notag
|E_K(x)|
\leq \min\bigg(\frac12,  \frac{1}{ \pi (2K+1) |\sin \pi x |}\bigg)
\leq  \min\bigg(\frac12,  \frac{1}{ 4\pi K \|  x \|}\bigg).
\ee
From this bound  we see that the series
$$
 -\sum_{k=1}^{\infty} \frac{\sin (2\pi x k)}{\pi k} 
$$
converges pointwise to $s(x)$ when  $x\notin\mathbb Z$. 
Moreover,  the series clearly converges to  $s(x)$ when  $x\in \mathbb Z$ as well, since then every term is zero. We may therefore summarize the above   in    
\begin{theorem}\label{T: S}
For   $\th\in \mathbb R$,
\be\label{S form}\notag
S(\th,\chi_D)=
 \sum_{\substack{j=1  }}^{2g} \big(s(-\th_j)- s(\th -\th_j)\big)
 =- \sum_{\substack{j=1  }}^{2g}  s(\th -\th_j)
 \ee
and  
\be\notag
S(\th, \chi_D) 
=   \sum_{j=1}^{2g} \sum_{k=1}^{\infty} \frac{\sin (2\pi( \th- \th_j )  k)}{\pi k} .
\ee
\end{theorem}
Note that the second formula in the theorem is what we would obtain formally from the first line of \eqref{L=Z_0} on taking \eqref{arg L 1} and basic properties of the $\th_j$'s  into account.

%
%
%
%
We can obtain   similar expressions for $\arg P_K(q^{-\frac12}e(\th), \chi_D)$.
From  \eqref{P def} and \eqref{arg PK}  
\be\notag
\begin{split}
\arg P_K(q^{-\frac12}e(\th),\chi_D)= &-\triangle_{\Gamma(\th)}\; \PKu \\
=&-\Im \; \sum_{k=1}^{ K}\;   \frac{e(k\th)}{k} 
\Bigg(   q^{-k/2}\sum_{\substack{f\, {\rm{monic}} \\ \deg f=k}} \Lambda(f) \chi_D(f) \Bigg).
\end{split}
\ee
Using \eqref{trace formula} to replace the expression in parentheses, we find that 
\be\label{arg P_K}\notag
\begin{split}
\arg P_K(q^{-\frac12}e(\th),\chi_D)
= &  \Im \; \sum_{k=1}^{ K}\;   \frac{e(k\th)}{k} 
\bigg(  \sum_{j=1}^{2g} e(-k\th_j)  \bigg) \\
=&   \sum_{j=1}^{2g} \bigg( \sum_{k=1}^{ K}\;   \frac{\sin(2\pi(\th-\th_j) k)}{k} 
  \bigg) .
\end{split}
\ee
By \eqref{sawtooth approx} this equals
 \be\notag
\begin{split}
   \pi \sum_{j=1}^{2g} \big( -s(\th-\th_j) + E_K(\th-\th_j)  \big).
\end{split}
\ee
Hence, we have  
\begin{theorem}\label{T: SK} For $K\geq 1$
 \be\label{arg P 2}\notag
\begin{split}
S_K(\th,\chi_D)=&    \sum_{j=1}^{2g} \big( -s(\th-\th_j) +  E_K(\th-\th_j)  \big) 
\end{split}
\ee
and 
\be\label{arg P 3}
S_K(\th,\chi_D)= \sum_{j=1}^{2g} \bigg( \sum_{k=1}^{ K}\;   
\frac{\sin(2\pi (\th-\th_j) k)}{\pi k} 
  \bigg),
\ee
where
\be\label{E term}\notag
|E_K(\th-\th_j)|
\leq \min\bigg(\frac12,  \frac{1}{ 4\pi K \| \th-\th_j\| }\bigg).
\ee
\end{theorem}

From Theorems~\ref{T: S} and \ref{T: SK} 
we immediately have

\begin{corollary}\label{C: S-SK} For $K\geq 1$
\be\label{ineq S-SK}
|S(\th,\chi_D)  -S_K(\th,\chi_D)| \leq \sum_{j=1}^{2g}\min\bigg(\frac12,  
\frac{1}{ 4\pi K \| \th-\th_j\| }\bigg).
\ee
\end{corollary}

 
\section{The counting function for the zeros of $\Lu$}

It is  a simple matter to count 
 the number of zeros  of $\Lu$ on an arc of the circle 
 $|u|=q^{-\frac12}$.  This was done by a slightly different method by Faifman and Rudnick~\cite{FR} for the  hyperelliptic ensemble  $\mathcal H_{2g+2, q}$ of even degree monic polynomials. We include  a proof because it is short.

Let $N(\th, \chi_D)$ denote  the number of zeros of $\Lu$ on the  
circular  arc $q^{-\frac12}e(\phi)$, \; $0\leq \phi\leq \th\leq 1$. 
That is, 
\be\notag
N(\th, \chi_D) =  \sum_{\th_j\leq \th}  1.
\ee
For the moment we assume that  $\th\neq \th_j, \, j=1,2, \ldots, 2g$.
Let   $C(\th)$ be the positively oriented contour 
consisting of the circular arc $q  e(\phi)$ from $\phi=0$ to $\phi=\th$, 
the radial segment $r e(\th)$ from $r=q $ to $r=q^{-2}$, the circular arc  
$ e(\phi)q^{-2}$ from $\phi=\th$ to $\phi=0$, and then the real segment  
 from $r=q^{-2}$ to $r=q $. Then
\be\label{N fnc}\notag
N(\th, \chi_D)= \frac{1}{2\pi} \triangle_{C} \arg \Lu. 
\ee
The change in argument along the real segment is zero. 
Along the outer circular arc from $ q $ to $q e(\th)$, and then  
along the radius from $q e(\th)$ to $ q^{-1/2}e(\th)$, we  use the functional equation
\be \notag 
\Lu = (qu^2)^g \mathcal{L}(1/qu, \chi_D)
\ee 
 to see that the change in argument equals $4\pi \th g +\triangle_{\gamma(\th)} \arg \Lu$,
 where $\gamma(\th)$ is the contour consisting of the 
circular arc from $u=q^{-2}$ to $u= q^{-2}e(-\th)$ and then continuing
along the radius $re(-\th)$ from $r=q^{-2}$ to $ q^{-\frac12}.$
But this is just minus the change in argument along $\Gamma(\th)$ (see just above \eqref{arg L 1}), which 
we defined to be $\arg\mathcal L(q^{-\frac12}e(\th), \chi_D)$. Thus
the change along the outer part of the contour equals \
$4\pi \th g+\arg\mathcal L(q^{-\frac12}e(\th), \chi_D)$. By \eqref{arg L 1} the remaining 
change in argument also equals $\arg\mathcal L(q^{-\frac12}e(\th), \chi_D)$.
Thus
\be \label{count zeros}
\begin{split}
N(\th, \chi_D)=& \frac{1}{2\pi} (4\pi\th g+2\arg\mathcal L(q^{-\frac12}e(\th), \chi_D)) \\
=& 2g\th  +\frac1\pi \arg\mathcal L(q^{-\frac12}e(\th), \chi_D)) \\
=& 2g\th+ S(\th, \chi_D).
\end{split}
\ee
If $\th=\th_j$ for some $j$, our convention \eqref{arg L at zeros} means that \eqref{count zeros} holds
in this case as well.
Thus we have proved
\begin{theorem} For all $\th\in[0,1]$
 we have
\be\label{count zeros 2}
N(\th, \chi_D)=2g\th +S(\th, \chi_D).
\ee
\end{theorem}
As a check of this formula we perform the following calculation.
According to Theorem~\ref{T: S}, if $0\leq \th\leq 1$,
\be\notag 
\begin{split}
S(\th,\chi_D)=
 \sum_{\substack{j=1  }}^{2g} \big(\{-\th_j\}- \{\th -\th_j\}\big).
\end{split}
\ee
Now $\{-\th_j\}=1-\th_j$, and 
$$
\{\th -\th_j\}=
\begin{cases}
\th-\th_j &\quad\hbox{ if}\quad 0<\th_j\leq \th,\\
1+\th-\th_j &\quad\hbox{ if}\quad \th < \th_j <1.
\end{cases}
$$
Hence the above is
\be \notag
\begin{split}
S(\th,\chi_D)=&
 \sum_{0<\th_j\leq \th} \big((1-\th_j)-(\th-\th_j)\big)
 +\sum_{\th<\th_j<1} \big((1-\th_j)-(1+\th-\th_j)\big)\\
 =& \sum_{0<\th_j\leq \th}  (1 - \th )-\sum_{\th<\th_j<1}\th \\
 =& N(\th, \chi_D)-2g\th,
\end{split}
\ee
which agrees with \eqref{count zeros 2}.

\section{Upper bounds  for $S(\th, \chi_D)$ and $S_K(\th, \chi_D)$}

Faifman~\cite{F} (see also Faifman and Rudnick~\cite{FR}, Proposition 5.1) 
has  shown that for the Hyperelliptic ensemble $\mathcal H_{2g+2, q}$,
\be\label{S bd}\notag
S(\th, \chi_D) \ll \frac{g}{\log_q g}.
\ee 
This is the analogue of the best known bound  for the order of 
$S(t)=(1/\pi)\arg\zeta(\frac12+it)$ on RH, namely,
\be\notag
S(t) \ll \frac{\log t}{\log \log t}.
\ee 
It is clear that the  methods of \cite{F} and \cite{FR}
apply to our ensemble  $\mathcal H_{2g+1, q}$ as well. 
In this section we   first give a  proof of this and then show that the same bound holds 
for $S_K(\th, \chi_D)$ if $K$ is sufficiently large with respect to $g$.
 
We use the following approximation result which we state without proof (see, for example, Montgomery~\cite{M}).
\begin{lemma}
Let $I=[\a,\beta]$ be an arc in $\mathbb T$ with length $\beta-\a<1$. Then for any positive integer $K$ there are trigonometric polynomials 
 $$  T^{\pm}(x) =\sum_{k=-K}^{K} a^{\pm}(k)e(kx)     $$
such that
\begin{itemize}
\item[(a)] $\displaystyle T^{-}(x) \leq \chi_I(x)\leq T^{+}(x)$ for all $x$,
\item[(b)] $\displaystyle \int_0^1 T^{\pm}(x) dx   = \beta-\alpha \pm \frac1{K+1}$.
\end{itemize}
\end{lemma}

\begin{theorem}
For $0<\th<1$,
\be\label{S bd 1}
S(\th, \chi_D) \ll \frac{g}{\log_q g}.
\ee
\end{theorem}

\begin{proof}
For $0<\th <1$ we have 
\be \notag
\begin{split}
 N(\th, \chi_D)=&\sum_{j=1}^{2g}\chi_{[0, \th]}(\th_j) \leq \sum_{j=1}^{2g}T^+(\th_j)   \\
=&2ga^{+}(0)\,+\,\sum_{\substack{k=-K\\k\neq 0}}^{K} a^{+}(k) \bigg(\sum_{j=1}^{2g} e(k\th_j)\bigg).
\end{split}
\ee
By \eqref{trace formula} and part (b) of the lemma, we thus see that
\be\label{S upper bd 1} 
\begin{split}
N(\th, \chi_D)\leq 
 2g\Big(\th+\frac1{K+1}\Big) -2\sum_{k=1}^{K} \frac{a^{+}(k)}{q^{k/2}}\bigg(\sum_{\deg f=k}\Lambda(f) \chi_D(f)\bigg).
\end{split}
\ee
Recall that if $f\in L^1(\mathbb T)$, then
$$
|\hat{f}(k) |=\bigg| \int_0^1   f(x) e(-kx) dx\bigg|
\leq \int_0^1   |f(x)| dx.
$$
Since 
$$
\hat{\chi}_{[0,\th]}(k) = e(-k\th/2)\frac{\sin\pi k\th}{\pi k}
$$
for $k\neq 0$, if we take $f=T^{+}-\chi_{[0,\th]}$, then again by (b) of the lemma
\be\notag
\Big|e(-k\th/2)\frac{\sin\pi k\th}{\pi k}- a^{+}(k) \Big|
\leq  \int_0^1  |\chi_{[0,\th]}(x)-T^+(x)| dx \leq \frac{1}{K+1}.
\ee
Thus,
$$
|a_k^{+}| \leq \frac{1}{K+1} + \bigg| \frac{\sin\pi k\th}{\pi k}\bigg|
\ll   \min\Big(\|\th\|, \frac1k\Big).
$$
From \eqref{S upper bd 1} it now follows that
\be \notag
\begin{split}
S(\th, \chi_D) =N(\th, \chi_D)-2g \th 
\leq &\frac {2g}{K} +O\bigg(\sum_{k=1}^{K}
\frac{ \min(\|\th\|, k^{-1}) }{q^{k/2}}\bigg(\sum_{\deg f=k}\Lambda(f)  \bigg)\bigg).
\end{split}
\ee
By the prime polynomial theorem, the sum in parentheses is $\ll q^k/k$.
Thus the second term on the right is
$$
\ll \sum_{k=1}^{K} \frac{ q^{k/2}  }{k^2} \ll q^{K/2}.
$$
Hence
$$
S(\th, \chi_D)   
\leq \frac {2g}{K} +O(q^{K/2}).
$$
The same argument using $T^{-}(x)$ instead of $T^{+}(x)$ leads to
$$
S(\th, \chi_D)   
\geq -\frac {2g}{K} +O(q^{K/2}).
$$
Thus
$$
S(\th, \chi_D)   
\ll \frac {g}{K} + q^{K/2}.
$$
Taking  $K=\log_q g$, we obtain \eqref{S bd 1}.
\end{proof}


In the case of the zeta function we expect much more to be true, and the  heuristic arguments  in Farmer, Gonek and Hughes~\cite{FGH} that indicate this also suggest that 
\be\label{S bd conj}\notag
S(\th, \chi_D) =O(\sqrt{g \log g}) \qquad 
\hbox{and}\qquad S(\th, \chi_D)=\Omega(\sqrt{g \log g}).
\ee
To accommodate any eventual improvements in the estimate,   we state the next result in terms of a general upper bound $\Phi(g)$ for $S(\th, \chi_D)$.  

\begin{theorem}\label{T: SK bd}
Suppose that 
$$
S(\th, \chi_D) \ll \Phi(g).
$$
Then for $K\geq (g\log g)/\Phi(g)$ we have
\be\label{S_K bd}
S_K(\th, \chi_D) \ll \Phi(g).
\ee
In particular, 
\be\notag
S_K(\th, \chi_D) \ll \frac{g}{\log_q g}
\ee
when $K\geq \log_q g \cdot\log g$
\end{theorem}

\begin{proof}
Set $\Delta=\Phi(g)/g$.
The right hand side of \eqref{ineq S-SK} is
\be\notag 
\begin{split}
\sum_{m=0}^{ [2g/\Delta]+1}
\sum_{\substack{j\\  m \Delta \leq \|\th-\th_j\|  <(m+1)\Delta}}
\min\bigg(\frac12,  \frac{1}{4 \pi K \|\th-\th_j\| }\bigg).
\end{split}
\ee
By \eqref{count zeros},  
$N(\th, \chi_D)=2g\th +O(\Phi(g))$. Hence, for each $m$ the number of terms in the 
inner sum over $j$ is $ 2g\Delta +O(\Phi(g)) \ll \Phi(g)$. The 
 $m=0$ term  therefore contributes $\ll  \frac12 \Phi(g)$, and the remaining terms contribute
\be\notag 
\begin{split}
\ll \frac1K\sum_{m=1}^{  [2g/\Delta]+1}
 \frac{  \Phi(g)}{m\Delta} 
 \ll \frac{\Phi(g)\log (2 g/\Delta)}{K \Delta}=\frac{g\log g}{K} .
\end{split}
\ee
Combining these estimates and taking $K\geq g\log g/\Phi(g)$, we obtain
\eqref{S_K bd}. The last assertion of the theorem is an immediate consequence of 
\eqref{S_K bd} and \eqref{S bd 1}. 
\end{proof}


\section{Discussion of a function related to $\Lu$}
We now introduce an auxiliary function $\Fu$ in order to study $\Lu$. 
For $u\in \mathbb C$
we define
\be\label{F def}
\Fu =\tfrac12\big(\Lu +(qu^2)^g \Lubar \big) .
  \ee 
Note that   $\Fu$ is not holomorphic  although it is harmonic. 

The reason we  introduce $\Fu$ is that, unlike $\Lu$ itself, we can model $\Fu$ on the 
closed disk $|u|\leq q^{-\frac12}$  along the lines of the approximate functional equation 
\eqref{App Fnc Equ},  using the  truncated Euler products $\PKu$.
By Theorems~\ref{L  approx P 1} and~\ref{L  approx P 2}, 
$\PKu$ is a good approximation of $\Lu$ when $K$ is large provided we are not too close to a zero of 
$\Lu$. This is inevitable because $\PKu$ can never vanish. Indeed, Theorem~\ref{L  approx P 2} indicates that 
the closer one is to a zero, the larger $K$ must be to retain a good approximation.  Thus, 
the approximation of $\Lu$ by $\PKu$ is least helpful where we  most need it---at the zeros. 
 
Fortunately, ``knowing'' $\Fu$ is  in many ways  the same  thing as   ``knowing'' $\Lu$.
For example, on the circle $|u|= q^{-\frac12}$, 
\be\label{F=L}
\Fu= \Lu.
\ee  
To see this observe that     $1/qu=\overline{u}$ when $|u|= q^{-\frac12}$, and that by the functional equation, 
$$
\Lu = (qu^2)^g \mathcal{L}(1/qu, \chi_D)= (qu^2)^g \Lubar .
$$
Using this in \eqref{F def}, we obtain \eqref{F=L}.

As another example consider the size of $\Lu$. From 
\eqref{F def} it is immediate that 
$$
\sup_{|u|\leq q^{-\frac12}}|\Fu| \leq \sup_{|u|\leq q^{-\frac12}}|\Lu|.
$$
In fact, however, the two quantities are equal. For   $\Fu$, being  harmonic,   must attain its maximum modulus  on the disc   
on the boundary. 
However,   $\Fu=\Lu$ on the boundary, so
$$
\sup_{|u|\leq q^{-\frac12}}|\Fu| = \sup_{|u|\leq q^{-\frac12}}|\Lu|.
$$

As a final example we prove 
\begin{theorem} The functions $\Fu$ and $\Lu$ have   the same zeros in $\mathbb C$. In particular, the Riemann hypothesis holds for $\Fu$.
\end{theorem}
\begin{proof}
 Since $\Fu= \Lu$,  on   $|u|= q^{-\frac12}$,  
both functions have the same zeros on this circle. 
Since $\Lu$ has no zeros anywhere else, to complete the proof we must
show that neither does $\Fu$.

Suppose, on the contrary, that $u_0$ is a zero of $\Fu$ with $|u_0|\neq q^{-\frac12}$. Since $\Lu$ has no zeros off the circle $|u|=q^{-\frac12}$, $u_0$ is not a 
zero of $\Lu$. We may therefore   write
\be\label{F=0}\notag
 0=\Fuzero  =\tfrac12\Luzero \bigg(1 + (qu_0^2)^g \frac{\Lubarzero}{\Luzero}\bigg).
\ee
The  only way the term  in parentheses on the right can vanish is if
$$
 \bigg| (qu_0^2)^g \frac{\Lubarzero}{\Luzero} \bigg|= 1.
$$
Now $|\Lubarzero/{\Luzero}|=1$, so this implies that $|u_0|= q^{-\frac12}$, a contradiction.
 \end{proof}
 
 
\section{A Model of $\Fu$}

Having shown that we can deduce information about $\Lu$ from information about  $\Fu$,
we now model   $\Fu$ using the  Euler product truncations $\PKu$.
We set
\be\notag
\FKu= \tfrac12\big(\PKu +(qu^2)^g \PKubar   \big),
\ee
where $\PKu$ is defined in \eqref{P def}.
Since $\PKu$  has no zeros, we see that  $\FKu=0$ if and only if
\be\label{zeros F_K}
1+ (qu^2)^g \frac{\PKubar}{\PKu} =0.
\ee
Since $|\PKubar/\PKu|=1$, this implies that $|u|= q^{-\frac12}$. Thus we have proved \\

\begin{theorem}[The Riemann hypothesis for $\FKu$] All zeros of $\FKu$ lie on 
$|u|= q^{-\frac12}$.
\end{theorem}

As  $\PKu$  approximates $\Lu$ well in $|u|\leq q^{-\frac12}$ when $K$ is large and $u$ 
is not too close to a zero of $\Lu$, $\FKu$ approximates $\Fu$.  
\begin{theorem}\label{FK approx F 1}
Let  $\displaystyle \s=\frac12+\frac{C\log_q g}{K }$ with $C\geq1$ and
  $K\geq 2C\log g$.  Let $|u|\leq  q^{-\s}$. Then 
\be\label{E: F approx 1}
\Fu=\FKu \Big(1+O\Big(\frac{ 1 }{Cg^{C-1}\log g}\Big)\Big).
\ee
On   the circle $u=q^{-\frac12}e(\th)$, if
 $\displaystyle \min_{1\leq j\leq 2g}\|\th-\th_j\| \geq  {c}/{2g }$
 with $c>0$ and $K  \geq g^2/c$, then
\be\label{E: F approx 2}
\Fu=\FKu \Big(1+O\Big(\frac{ g^2}{cK }\Big)\Big).
\ee
 \end{theorem}
\begin{proof}
By Theorem~\ref{L approx P 1} and the definition of $\Fu$,
\be\notag 
\begin{split}
\Fu= \tfrac12\big(\PKu +(qu^2)^g \PKubar   \big)
\Big(1+O\Big(\frac{ 1 }{Cg^{C-1}\log g}\Big)\Big).
\end{split}
\ee
Equation \eqref{E: F approx 1} now follows from the definition of $\FKu$.
The proof of \eqref{E: F approx 2} is the same except that one uses
Theorem~\ref{L approx P 2}.
\end{proof} 
 \begin{corollary}\label{C: lim F_K} For  $|u|\leq q^{-\frac12}$, $u\neq u_j, \, j=1,2, \ldots, 2g$,
$$
 \lim_{K\to\infty}\FKu\to \Fu.
$$
\end{corollary}


\section{The zeros of $\FKu$}

Since $\FKu$ is a good approximation of $\Fu$, one  wonders  whether their zeros  are close to one another or whether there are other connections between them. We have seen that both functions satisfy the Riemann hypothesis, so that is a good start.
As previously, we write $u_j =q^{-\frac12}e(\th_j) , j=1, 2, \ldots, 2g$,  for the 
zeros of $\Lu$, where $0\leq \th_1\leq \th_2\leq  \ldots \leq \th_{2g}<1$. 
We denote the zeros of  $\FKu$
by  $v_j=q^{-\frac12}e(\phi_j), \, j=1, 2, \ldots $,   with
$0\leq \phi_1\leq \phi_2\leq  \ldots   <1$, leaving open for now the question of their number.

By \eqref{zeros F_K},
 a necessary and sufficient condition for $v_j=q^{-\frac12}e(\phi_j)$ 
 to be a zero of $\FKu$ is that 
\be\label{}\notag
1+ (q v_j^2)^g \frac{P_K(\overline{v}_j,\chi_D)}{P_K(v_j,\chi_D)} =0.
\ee
This is equivalent to 
\be\notag
e^{4\pi i g\phi_j +2i\arg P_K(v_j,\chi_D)} =-1,
\ee
or
\be\notag
2  g\phi_j +\frac1\pi \arg P_K(v_j,\chi_D)=2  g\phi_j +S_K(\phi_j,\chi_D) \equiv  \tfrac12 \pmod{1}.
\ee
Now as $\phi$ varies from $0$ to $1$,
 the graph of the continuous  curve  
\be\label{def f_K}
f_K(\phi) = 2 g\phi + S_K ( \phi, \chi_D)
\ee
traverses a vertical distance  greater than or equal to $f_K(1)- f_K(0) =2g-0=2g$.
Thus it intersects at least $2g$ of the horizontal lines $y=  k+\frac12,\, k \in \mathbb Z$, 
possibly more than once.
We let these values be $\phi_1, \phi_2, \ldots $
in increasing order. Then the points $v_j=q^{-1/2}e(\phi_j)$ are  the  \emph{distinct} 
zeros of $\FKu$.
Thus, $\Fu$ has $2g$ zeros, counting multiplicities, and   $\FKu$ has at least $2g$  distinct zeros.  
Similarly, we see that the number of    zeros of $\FKu$ on any arc 
$u=q^{-\frac12}e(\phi),\, 0\leq \phi\leq \th $, where   $0\leq  \th<1$, is   
\be\notag 
\begin{split}
N_K (\th, \chi_D) \geq &2g\th + S_K ( \th, \chi_D) +O(1).
\end{split}
\ee
 
 Combining this with Theorem~\ref{T: SK bd} we obtain
 \begin{theorem}\label{T: N_K lower bd}
 Let $K\geq g\log g/\Phi(g)$. Then
\be\notag 
\begin{split}
N_K (\th, \chi_D) \geq  2g\th + O(\Phi(g)).
\end{split}
\ee
 \end{theorem}
 
Next we show  that  the zeros of $\FKu$ are close to those of $\Fu$ when $K$ is  large.
We saw that $\FKu$ has a zero at $u=q^{-\frac12}e(\th)$ if and only if 
\be\label{f_K cong}
f_K(\th)=2g\th+S_K(\th, \chi_D) \equiv  \tfrac12 \pmod{1}.
\ee
Thus
\be\label{f_K formula 1}\notag
\begin{split}
f_K(\th)=&2g\th+S(\th, \chi_D) +(S_K(\th, \chi_D) -S(\th, \chi_D)\\
=&N(\th, \chi_D)+(S_K(\th, \chi_D) -S(\th, \chi_D) ).
\end{split}
\ee
Suppose now that $\th_i$ and $\th_{i+1}$ are arguments corresponding to distinct consecutive zeros of $\Fu$, and let $0<\Delta<\frac12 (\th_{i+1}-\th_i)$. Then on the interval $I =[\th_i+\Delta, \th_{i+1}-\Delta]$, $N(\th, \chi_D)$ is an integer. Thus,
if $|S_K(\th, \chi_D) -S(\th, \chi_D)|<\frac12$ on $ I$, then
\eqref{f_K cong} cannot hold. By  \eqref{ineq S-SK}, if $\th\in I$ we have
 
 \be\notag
\begin{split}
|S(\th,\chi_D)  -S_K(\th,\chi_D)| 
\leq& \sum_{j=1}^{2g} \min\bigg(\frac12,  \frac{1}{ 4\pi K \| \th-\th_j \| }\bigg) \\
\leq& \frac{g}{2\pi  K\Delta} .
\end{split}
\ee
 Therefore,   if $K>g/\pi\Delta$, 
then $S_K(\th, \chi_D) -S(\th, \chi_D)<\frac12$ on $ I$.
We have thus proved
\begin{theorem}\label{T: zeros cluster}
Let $\th_i$ and $\th_{i+1}$  correspond  to distinct consecutive zeros of $\Fu$, and let $0<\Delta<\frac12 (\th_{i+1}-\th_i)$. Then if  $K>g/\pi\Delta$,  $\FKu$ has no zero on the 
interval $  I =[\th_i+\Delta, \th_{i+1}-\Delta]$. In particular, the zeros of $\FKu$ cluster around the zeros of $\Fu$  as $K\to\infty$.
\end{theorem}

Our last theorem concerns the  simplicity of zeros of $\FKu$.
We may write
\be\label{F_K formula}\notag
\mathcal F_K(q^{-\frac12}e(\th), \chi_D) 
=\tfrac12P_K(q^{-\frac12}e(\th), \chi_D)(1+e(f_K(\th)))
\ee
with $f_K(\th)$ defined in \eqref{def f_K}. Recall that $\th$ corresponds to a zero of $\FKu$ 
if and only if \eqref{f_K cong} holds. This zero will be simple if and only if
\be\notag
\begin{split}
\frac{d}{d\th} \mathcal F_K(q^{-\frac12}e(\th), \chi_D) \neq 0,
\end{split}
\ee
which is easily seen to be equivalent to
\be\notag
\begin{split}
   \frac{df_K(\th)}{d\th}\neq 0.
\end{split}
\ee
By  \eqref{def f_K} and \eqref{arg P 3},
\be\notag 
\begin{split}
   \frac{df_K(\th)}{d\th} =2g+2
\sum_{j=1}^{2g} \bigg( \sum_{k=1}^{ K}\;  \cos(2\pi(\th-\th_j) k) 
  \bigg) .
\end{split}
\ee
The right hand side is a trigonometric polynomial in $\th$ of degree $K$ so it has at 
most $2K$ zeros. Now, by Theorem~\ref{T: N_K lower bd},  if $K\geq g\log g /\Phi(g)$, then   $\FKu$ has $\geq 2g(1+o(1))$ zeros. Thus, if we also have $K=o(g)$, then at most $o(g)$ of these will be multiple. Taking $\Phi(g)= g/\log_q g$, we deduce 

\begin{theorem} If \  $\log g \log_q g  \leq K=o(g)$, then almost all zeros of $\FKu$ are simple.
\end{theorem}



{
}

\end{document}